\documentclass[leqno]{amsart}

\usepackage{amsmath, amssymb, amsthm}
\usepackage{hyperref}
\usepackage{cleveref}
\usepackage{chngcntr}
\usepackage{thmtools}
\usepackage{thm-restate}
\usepackage{fancyhdr}
\usepackage{dsfont}
\usepackage[english]{babel}
\usepackage[utf8]{inputenc}
\usepackage{commath}
\usepackage{marvosym}
\usepackage{mathtools}
\usepackage{cite}
\usepackage{thmtools}
\usepackage{thm-restate}
\usepackage{soul}
\usepackage{tikz-cd}
\usepackage[margin=1.5in]{geometry}
\usepackage{comment}
\usepackage{enumitem}

\newtheorem{theorem}{Theorem}[section]

\newtheorem{corollary}[theorem]{Corollary}
\newtheorem{proposition}[theorem]{Proposition}

\newtheorem{lemma}[theorem]{Lemma}

\newtheorem*{theorem*}{Theorem}
\newtheorem*{corollary*}{Corollary}
\newtheorem*{lemma*}{Lemma}
\newtheorem*{proposition*}{Proposition}

\theoremstyle{definition}
\newtheorem{definition}[theorem]{Definition}
\newtheorem{remark}[theorem]{Remark}
\newtheorem{example}[theorem]{Example}

\newtheorem{setup}[theorem]{Setup}

\newtheorem{chunk}[theorem]{}

\newtheorem{thmx}{Theorem}[section]

\newcommand{\thick}[3]{\operatorname{\mathsf{Thick}}_{#1}^{#2}#3}
\newcommand{\cat}[2]{\mathsf{#2}(#1)}

\newcommand{\catd}[1]{\mathsf{D}(#1)}
\newcommand{\catdbf}[1]{\mathsf{D_{b}^{f}}(#1)}
\newcommand{\catdbfl}[1]{\mathsf{D_{b}^{fl}}(#1)}

\newcommand{\mfm}{\mathfrak{m}}
\newcommand{\ext}[4]{\operatorname{Ext}_{#1}^{#2}(#3,#4)}
\newcommand{\tor}[4]{\operatorname{Tor}^{#1}_{#2}(#3,#4)}
\newcommand{\gdim}[2]{\operatorname{Gdim}_{#1}#2}
\newcommand{\projdim}[2]{\operatorname{projdim}_{#1}#2}
\newcommand{\injdim}[2]{\operatorname{injdim}_{#1}#2}
\newcommand{\h}[2]{\operatorname{H}_{#1}(#2)}
\newcommand{\HR}{\h{0}{R}}
\newcommand{\Hom}[3]{\operatorname{Hom}_{#1}(#2,#3)}
\newcommand{\rhom}[3]{\operatorname{\mathsf{R}Hom}_{#1}(#2,#3)}
\newcommand{\ltime}[3]{#1\otimes_{#2}^{\mathsf{L}}#3}

\newcommand{\pseries}[2]{\mathsf{P}_{#1}^{#2}(t)}
\newcommand{\bseries}[2]{\mathsf{I}^{#1}_{#2}(t)}
\newcommand{\gzero}{\mathcal{G}_{0}}
\newcommand{\amp}{\operatorname{amp}}
\newcommand{\ann}[2]{\operatorname{ann}_{#1}#2}

\author{Zachary Nason}
\address{University of Nebraska Lincoln, NE 68588. U.S.A.}
\email{znason2@huskers.unl.edu}
\urladdr{https://zach-nason.github.io/}

\author{Andrew J. Soto Levins}
\address{Texas Tech University, TX 79409. U.S.A.}
\email{ansotole@ttu.edu}
\urladdr{https://sites.google.com/view/andrewjsotolevins}

\author{Ryan Watson}
\address{University of Nebraska Lincoln, NE 68588. U.S.A.}
\email{rwatson9@huskers.unl.edu}
\urladdr{https://rawatson1997.github.io/}

\title{Quasi-Gorenstein Morphisms of commutative local dg-algebras}

\date{\today}
\subjclass[2020]{Primary: 13B10, 13D09, 13H10, 16E35, 16E45}
\keywords{dg-algebra, exact zero divisor, Gorenstein projective dg-modules, quasi-Gorenstein, virtually G small}

\begin{document}

\begin{abstract}
We introduce quasi-Gorenstein morphisms of commutative local dg-algebras and use a Gorenstein version of the virtually small property to characterize them, a result which is new even for homomorphisms of local rings. In a different direction, we characterize exact sequences in a noetherian local ring, in the sense of Avramov, Henriques, and \c{S}ega, in terms of quasi-Gorenstein morphisms involving Koszul complexes.
\end{abstract}
\maketitle

\section{Introduction}

In \cite{Avramov/Foxby:1997} Avramov and Foxby introduced quasi-Gorenstein homomorphisms of noetherian local rings. These are the homomorphisms that capture the Gorenstein property: for a noetherian local ring $R$ with residue field $k$, $R\rightarrow k$ is quasi-Gorenstein if and only if $R$ is Gorenstein. Moreover, if a local homomorphism $\varphi\colon R\rightarrow S$ of noetherian local rings is quasi-Gorenstein, then $R$ is Gorenstein if and only if $S$ is Gorenstein.

In this paper we introduce quasi-Gorenstein morphisms of commutative local dg-algebras, extending work of Avramov and Foxby. We first review necessary background for the paper in \Cref{S_Background}, and in \Cref{S_QGMorphismsDGAS} we prove dg-analogs of results about quasi-Gorenstein ring homomorphisms that will be used throughout the paper.

The following is our first main result - it is new even when both $R$ and $S$ are rings. 
\begin{thmx}[\Cref{P_FiniteGdimAcrossMorphisms} and \Cref{T_GvirtuallysmallnessascendsimpliesquasiGor}] Let $\varphi\colon R\rightarrow S$ be a morphism of commutative local dg algebras with $S_{0}$ a noetherian local ring and $\gdim{R}{S}<\infty$. Then the following are equivalent
\begin{enumerate}
    \item $\varphi$ is quasi-Gorenstein.
    \item Every dg $S$-module $M$ with bounded and degreewise finite length homology that has finite Gorenstein dimension over $R$ is virtually $G$ small over $S$.
\end{enumerate}
\end{thmx}

This theorem parallels results of Briggs, Iyengar, Letz, and Pollitz \cite[Theorem 3.2]{Briggs/Iyengar/Letz/Pollitz:2022} and of Dwyer, Greenlees, and Iyengar \cite[Theorem 9.1]{Dwyer/Greenlees/Iyengar:2006} that show if $\varphi\colon R\rightarrow S$ is a surjective homomorphism of noetherian rings of finite flat dimension, then $\varphi$ is locally complete intersection if and only if any $S$-complex with finite length homology that is proxy small as an $R$-complex is also proxy small over $S$.

The following is our second main result. Recall that a sequence $x_{1},\dots,x_{n}$ in a noetherian local ring $A$ is exact if for $1\leq i\leq n-1$, $x_{i}$ is $A/(x_{1},\dots,x_{i+1})$-regular or is an exact zero divisor on $A/(x_{1},\dots,x_{i+1})$.
\begin{thmx}[\Cref{P_ExactSequences}] Let $A$ be a noetherian local ring, let $x_{1},\dots,x_{n}$ be a sequence of elements in $A$, and let $R$ be the Koszul complex on $x_{1},\dots,x_{n}$. Then $R\rightarrow \h{0}{R}$ is quasi-Gorenstein and $\gdim{A/(x_{1},\dots,x_{i-1})}{A/(x_{1},\dots,x_{i})} < \infty$ for $1\leq i\leq n$ if and only if the sequence $x_{1},\dots,x_{n}$ is exact.
\end{thmx}

In particular, this theorem gives a sufficient condition for an ideal to be quasi-complete intersection: by \cite[Theorem 3.7]{Avramov/Henriques/Sega:2013} ideals generated by exact sequences are quasi-complete intersection. Moreover, \Cref{P_ExactSequences} and \Cref{C_DGAexampleWithExactElement} give examples where $R\rightarrow \HR$ is quasi-Gorenstein and $R$ is a Koszul complex. We provide more examples of quasi-Gorenstein morphisms of dg-algebras in \Cref{S_MoreExamples}

We end the paper with \Cref{S_GorProjModsKoszul} by using results in previous sections to give ways of producing nonfree Gorenstein projective dg-modules over a Koszul complex on a sequence of exact zero divisors. In particular, if the dg-algebra is local-Cohen-Macaulay, we give examples of nonfree maximal local-Cohen-Macaulay dg-modules, see \Cref{R_ExactZeroDivisorMCMdgModules}.

In this paper we work under the following setup.
\begin{setup} \label{S_dgAlgebraSetup} When we say $R$ is a dg-algebra, we mean a nonnegative strongly graded-commutative dg-algebra such that $(\h{0}{R},\mfm,k)$ is a noetherian local ring and such that $\h{}{R}$ is finitely generated over $\h{0}{R}$. When we say $\varphi\colon R\rightarrow S$ is a morphism of dg-algebras, we mean $\varphi$ is a morphism of dg-algebras such that the induced map $\h{0}{R}\rightarrow \h{0}{S}$ is a local homomorphism of noetherian local rings. 
\end{setup}

\section*{Acknowledgments}
We thank Lars Christensen, Nawaj KC, Tom Marley, and Josh Pollitz for many helpful conversations throughout this project.

\section{Background} \label{S_Background}
In this section we review necessary background for the paper. We work in the derived category $\cat{R}{D}$ of dg $R$-modules, and focus on the subcategory $\catdbf{R}$, which consists of all dg $R$-modules $M$ such that $\h{}{M}$ is bounded and degreewise finitely generated over $\h{0}{R}$. We write $\simeq$ for isomorphisms in the derived category, and write $\sim$ for isomorphisms in the derived category up to a shift. We reserve the symbol $\cong$ to denote isomorphisms of dg-modules. The supremum and infimum of a dg-module $M$ are
\[\sup{M}=\sup\{n \mid \h{n}{M}\neq 0\}\quad\quad \inf{M}=\inf\{n \mid \h{n}{M}\neq 0\}.\]
The amplitude of $M$ is $\amp{M}=\sup{M}-\inf{M}$.

\begin{chunk} A subcategory of $\catdbf{R}$ is thick if it is closed under taking shifts, finite direct sums, summands, and triangles. The thick closure of a dg-module $M$, denoted by $\thick{R}{}{M}$, is the smallest thick subcategory containing $M$.
\end{chunk}

\begin{chunk} \label{C_ProjDimInjDim} Let $M$ be a dg-module in $\catdbf{R}$. Following Bird, Shaul, Sridhar, and Williamson \cite{Bird/Shaul/Sridhar/Williamson:2025}, the \textit{projective dimension} of a dg-module $M$ is
\begin{align*}
        \projdim{R}{M}:=\inf\{n\in \mathbb{Z} \mid \ext{R}{i}{M}{N}=0  &\text{ for any } N\in \catdbf{R} \\
        &\text{ and any } i>n-\inf{N} \}.
\end{align*}    
The following are equivalent by \cite[Proposition 2.2]{Minamoto:2021a} and \cite[Theorem 2.22]{Minamoto:2021b}:
\begin{enumerate}
    \item $\projdim{R}{M}<\infty$.
    \item $\tor{R}{n}{k}{M}=0$ for $n\gg 0$.
    \item $M$ is in $\thick{R}{}{R}$.
\end{enumerate}
We also will use the Poincaré series of $M$, which is defined to be 
\[\pseries{M}{R} = \sum\dim_{k}\tor{R}{n}{k}{M}t^{n}.\]
Similarly, the \textit{injective dimension} of $M$ is defined to be
\begin{align*}
        \injdim{R}{M}:=\inf\{n\in \mathbb{Z} \mid \ext{R}{i}{N}{M}=0  &\text{ for any } N\in \catdbf{R} \\
        &\text{ and any } i>n+\sup{{N}} \}.
\end{align*}
By \cite[Corollary 2.31]{Minamoto:2021a} $\injdim{R}{M}<\infty$ if and only if $\ext{R}{n}{k}{M} = 0$ for $n \gg 0$. The Bass series of $M$ is defined to be
\[\bseries{M}{R} = \sum\dim_{k}\ext{R}{n}{k}{M}.\]
\end{chunk}

\begin{chunk} A Gorenstein dimension for dg-modules was defined by Hu, Yang, and Zhu in \cite{Hu/Yang/Zhu:2025}. Let $M\in\catdbf{R}$ be a derived reflexive dg-module.
\begin{enumerate}
    \item The \textit{Gorenstein dimension} of $M$, denoted by $\gdim{A}{M}$, is $-\inf{\rhom{R}{M}{R}}$.
    \item The dg-module $M$ is in the class $\mathcal{G}$ if either $\gdim{R}{M}=\inf{M}$ or $M=0$, and $M$ is in the class $\gzero$ if it is an object of $\mathcal{G}$ with the additional properties that $\amp{M}\geq \amp{A}$ and $\gdim{R}{M}=\inf{M}=0$, or $M=0$. The dg-modules in $\gzero$ are called Gorenstein projective.
\end{enumerate}
By \cite[Theorem 1.2]{Hu/Yang/Zhu:2025} $\gdim{R}{M}<\infty$ if and only if $M$ is in $\thick{R}{}{\gzero}$.
\end{chunk}

\begin{chunk} A dg-module $D$ is a \textit{dualizing} dg-module over $R$ if $D$ has finite injective dimension and if the canonical morphism
\[D\rightarrow \rhom{R}{D}{D}\]
is an isomorphism in the derived category. The dg-algebra $R$ is \textit{Gorenstein} if $\injdim{R}{R}<\infty$, see \cite{Frankild/Iyengar/Jorgensen:2003, Frankild/Jorgensen:2003}. The following are equivalent, see \cite[Theorem 1.2]{Hu/Yang/Zhu:2025}:
\begin{enumerate}
    \item $R$ is Gorenstein.
    \item $R$ is a dualizing dg-module for $R$.
    \item $\gdim{R}{M}<\infty$ for all $M$ in $\catdbf{R}$.
    \item $\gdim{R}{k}<\infty$.
\end{enumerate}
\end{chunk}

\section{Quasi-Gorenstein morphisms of dg-algebras} \label{S_QGMorphismsDGAS}
In this section we prove dg-analogs of results about quasi-Gorenstein ring homomorphisms that will be used throughout the paper. Gorenstein and quasi-Gorenstein morphisms of rings were first studied in \cite{Avramov/Foxby:1992, Avramov/Foxby:1997}, and Gorenstein morphisms of dg-algebras were defined in \cite{Frankild/Jorgensen:2003}.
\begin{definition} \label{quasiGorfordgalgebras} Let $\varphi\colon R\rightarrow S$ be a morphism of dg-algebras with $\gdim{R}{S}<\infty$. We say $\varphi$ is \textit{quasi-Gorenstein} if $\rhom{R}{S}{R}$ is isomorphic in $\catd{S}$ to a shift of $S$.
\end{definition}

\begin{remark} Let $\varphi\colon R\rightarrow S$ be a module finite local homomorphism of noetherian local rings with $\gdim{R}{S}<\infty$. By \cite[Theorem 7.8]{Avramov/Foxby:1997}, $\varphi$ is quasi-Gorenstein if and only if $\varphi$ is quasi-Gorenstein in the sense of Definition \ref{quasiGorfordgalgebras}.
\end{remark}

The following was proved in \cite[Theorem 17.1.23]{Christensen/Foxby/Holm:2024} when $R$ is a ring.
\begin{lemma} \label{L_checkwhenMhasfiniteGdim} Let $R$ be a dg-algebra and let $M$ a dg-module in $\catdbf{R}$. If there exists an isomorphism
\[\phi\colon M\simeq \rhom{R}{\rhom{R}{M}{R}}{R},\]
then the canonical morphism
\[\delta_R^M\colon M\rightarrow \rhom{R}{\rhom{R}{M}{R}}{R}\]
is an isomorphism.
\end{lemma}

\begin{proof}
We adapt the proof of \cite[Proposition 17.1.23]{Christensen/Foxby/Holm:2024} to the dg-setting. Let $(-)^\dagger$ be the functor $\rhom{R}{-}{R}$. The functor $(-)^\dagger$ is self-adjoint on the right with unit $\delta_R^-$. Consider the composition of morphisms $(\delta_R^{M})^\dagger \circ \delta_R^{M^\dagger}$. Through the zig-zag identities, we have that $(\delta_R^{M})^\dagger \circ \delta_R^{M^\dagger} = 1_R^{M^\dagger}$, and so we have that $\h{}{(\delta_R^{M})^\dagger} \circ \h{}{\delta_R^{M^\dagger}} = \h{}{1_R^{M^\dagger}}$ by applying the homology functor. By \cite[Theorem 17.1.22]{Christensen/Foxby/Holm:2024}, we have that $\h{}{\delta_R^{M^\dagger}}$ is an isomorphism of $\h{}{R}$-complexes, which, by \cite[Corollary 7.2.13]{Yekutieli:2020} implies that $\delta_R^{M^\dagger}$ is an isomorphism in $\cat{R}{D}$. Through the same reasoning, we obtain that $\delta_R^{M^{\dagger\dagger}}$ is an isomorphism in $\cat{R}{D}$.

As biduality $\delta_R^-$ is a natural transformation, we have the commutative diagram
\begin{center}
\begin{tikzcd}
M \arrow[r, "\phi"]\arrow[d, "\delta_R^M"] & M^{\dagger\dagger}\arrow[d, "\delta_R^{M^{\dagger\dagger}}"] \\
M^{\dagger\dagger} \arrow[r, "\phi^{\dagger\dagger}"] & M^{\dagger\dagger\dagger\dagger}
\end{tikzcd}
\end{center}
Since $\phi$, $\phi^{\dagger\dagger}$, and $\delta_R^{M^{\dagger\dagger}}$ are isomorphisms, the commutativity of the diagram forces $\delta_R^M$ to be an isomorphism as well.
\end{proof}

The following was proved in \cite[Theorem 7.11]{Avramov/Foxby:1997} when both $R$ and $S$ are rings and $M$ is a finitely generated $S$ module.
\begin{proposition} \label{P_FiniteGdimAcrossMorphisms} Let $\varphi\colon R\rightarrow S$ be a quasi-Gorenstein morphism of dg-algebras and let $M$ be in $\catdbf{S}$. Then
\[\gdim{R}{M} = \gdim{R}{S} + \gdim{S}{M}.\]
In particular, $\gdim{R}{M}<\infty$ if and only if $\gdim{S}{M}<\infty$.
\end{proposition}

\begin{proof}
Let $n=\gdim{R}{S}$. Since $\varphi$ is quasi-Gorenstein, $S$ is isomorphic to $\rhom{R}{S}{\Sigma^{n}R}$. We now have the following isomorphisms
\begin{align*}
\rhom{S}{M}{S} \simeq \rhom{S}{M}{\rhom{R}{S}{\Sigma^{n}R}} &\simeq \rhom{R}{\ltime{M}{S}{S}}{\Sigma^{n}R} \\
&\simeq \rhom{R}{M}{\Sigma^{n}R}
\end{align*}
and so $\rhom{S}{M}{S}$ is in $\catdbf{S}$ if and only if $\rhom{R}{M}{\Sigma^{n}R}$ is in $\catdbf{R}$. In particular, if $M$ has finite Gorenstein dimension over $R$ and $S$, then
\[\gdim{R}{M} = \gdim{R}{S} + \gdim{S}{M}.\]
We have the following isomorphisms:
\begin{align*}
\rhom{R}{\rhom{R}{M}{R}}{R} &\simeq \rhom{R}{\rhom{R}{M}{\Sigma^{n}R}}{\Sigma^{n}R} \\
&\simeq \rhom{R}{\rhom{S}{M}{S}}{\Sigma^{n}R} \\
&\simeq \rhom{R}{\rhom{S}{M}{S}\ltime{}{S}{S}}{\Sigma^{n}R} \\
&\simeq \rhom{S}{\rhom{S}{M}{S}}{\rhom{R}{S}{\Sigma^{n}R}} \\
&\simeq \rhom{S}{\rhom{S}{M}{S}}{S}.
\end{align*}
Now applying \Cref{L_checkwhenMhasfiniteGdim} and \cite[Corollary 7.2.13]{Yekutieli:2020}, and the same reasoning in \cite[Proposition 12.3.19]{Christensen/Foxby/Holm:2024}, shows that $M$ is derived reflexive over $R$ if and only if it is derived reflexive over $S$, and so the proof is finished. 
\end{proof}

The next two results were proved in \cite[8.10.a]{Avramov/Foxby:1997} and \cite[8.9]{Avramov/Foxby:1997}, respectively, when $Q$, $R$, and $S$ are rings.
\begin{proposition} \label{P_PhiPsiPsiGorImpliesPhiGor} Let $Q\xrightarrow{\psi} R\xrightarrow{\varphi} S$ be morphisms of dg-algebras. If $\varphi\psi$ and $\psi$ are quasi-Gorenstein, then $\varphi$ is quasi-Gorenstein.
\end{proposition}

\begin{proof}
Note that $S$ has finite Gorenstein dimension over $R$ by \Cref{P_FiniteGdimAcrossMorphisms} since $S$ has finite Gorenstein dimension over $Q$. Also, we have the following isomorphisms in $\catd{S}$
\begin{align*}
\rhom{R}{S}{R} \sim \rhom{R}{S}{\rhom{Q}{R}{Q}} &\simeq \rhom{Q}{\ltime{S}{R}{R}}{Q} \\
&\simeq \rhom{Q}{S}{Q} \\
&\sim S.
\end{align*}
The first and fourth isomorphisms hold as both $\psi$ and $\varphi\psi$ are quasi-Gorenstein, and the second isomorphism is in $\catd{S}$ by \cite[Proposition 12.10.12]{Yekutieli:2020}.
\end{proof}

\begin{proposition} \label{P_CompositionQuasiGor} Let $Q\xrightarrow{\psi} R\xrightarrow{\varphi} S$ be morphisms of dg-algebras. If $\psi$ and $\varphi$ are quasi-Gorenstein, then $\varphi\psi$ are quasi-Gorenstein.
\end{proposition}

\begin{proof}
Note that $S$ has finite Gorenstein dimension over $Q$ by \Cref{P_FiniteGdimAcrossMorphisms} since $S$ has finite Gorenstein dimension over $R$. Also, we have the following isomorphisms in $\catd{S}$
\begin{align*}
    \rhom{Q}{S}{Q}  \simeq \rhom{Q}{\ltime{S}{R}{R}}{Q} &\simeq \rhom{R}{S}{\rhom{Q}{R}{Q}} \\
                    &\sim \rhom{R}{S}{R} \\
                    &\sim S
\end{align*}
where the second isomorphism is in $\catd{S}$ by \cite[Proposition 12.10.12]{Yekutieli:2020}, and the third and fourth isomorphisms hold as both $\varphi$ and $\psi$ are quasi-Gorenstein.
\end{proof}

The following was proved in \cite[7.7.2]{Avramov/Foxby:1997} when both $R$ and $S$ are rings, and it was proved for Gorenstein morphisms of dg-algebras in \cite[Theorem 3.6 and
Theorem 3.10]{Frankild/Jorgensen:2003}.
\begin{proposition} \label{P_RGorAndSGorSimultaneously} Let $\varphi\colon R\rightarrow S$ be a quasi-Gorenstein morphism of dg-algebras. Then $R$ is Gorenstein if and only if $S$ is Gorenstein.
\end{proposition}

\begin{proof}
Let $k$ be the residue field of $\h{0}{R}$ and let $l$ be the residue field of $\h{0}{S}$. Since $S$ is in $\catdbf{R}$, $l$ is a finite dimensional $k$-vector space. This gives the sixth isomorphism below:
\begin{align*}
\rhom{S}{l}{S} \sim \rhom{S}{l}{\rhom{R}{S}{R}} &\simeq \rhom{R}{\ltime{l}{S}{S}}{R} \\
&\simeq \rhom{R}{l}{R} \\
&\simeq \rhom{R}{\ltime{l}{k}{k}}{R} \\
&\simeq \rhom{k}{l}{\rhom{R}{k}{R}} \\
&\simeq \bigoplus_{\text{finite}} \rhom{R}{k}{R}.
\end{align*}
The first isomorphism holds as $\varphi$ is quasi-Gorenstein. It is now clear that $R$ has finite injective dimension over itself if and only if $S$ does.
\end{proof}

\section{Virtually G Small} \label{S_VirtuallyGSmall}
In this section we use a Gorenstein version of virtually small to give a sufficient condition for a morphism of dg-algebras to be quasi-Gorenstein, which is new even for homomorphisms of local rings. Virtually G small complexes were first studied by Dywer, Greenlees, and Iyengar \cite{Dwyer/Greenlees/Iyengar:2006}.
\begin{definition} Let $R$ be a dg-algebra. A dg-module $M\in\catdbf{R}$ is virtually $G$ small if there exists $F_{M}\in\thick{R}{}{M}$ so that $\gdim{R}{F_{M}}<\infty$.
\end{definition}

\begin{chunk} \label{C_KoszulObjectStuff} Let $R$ be a dg-algebra and let $x$ be an element in $\h{0}{R}$. Under the identification 
\[\h{0}{R} \simeq \h{0}{\rhom{R}{R}{R}} \simeq \Hom{\catd{R}}{R}{R},\]
$x$ corresponds to an endomorphism of $R$ in $\catdbf{R}$. Denote the cone of this morphism by $R/\!/x$. In \cite[Section 2.2]{Minamoto:2021a} and \cite[Section 5]{Shaul:2020} it is shown that $R/\!/x$ has the structure of a dg-algebra. If $\underline{x}=x_{1},\dots,x_{n}$ is a sequence in $\h{0}{R}$, then $R/\!/\underline{x}$ is defined inductively as $(R/\!/(x_{1},\dots,x_{n-1}))/\!/x_{n}$. Moreover, we have $\h{0}{R/\!/\underline{x}}\cong \h{0}{R}/\underline{x}$ and that $\h{}{R/\!/\underline{x}}$ is finitely generated over $\h{0}{R/\!/\underline{x}}$. In particular, $\h{}{R/\!/\underline{x}}$ has finite length if $\h{0}{R/\!/\underline{x}}$ is artinian.
\end{chunk}

A dg $R$-module $M$ is in $\catdbfl{R}$ if $M$ is in $\catdbf{R}$ and $\h{}{M}$ has finite length as an $\h{0}{R}$-module.
\begin{remark} Let $R$ be a dg-algebra. Then the following are equivalent
\begin{enumerate}
    \item $R$ is Gorenstein.
    \item Every $M\in\catdbf{R}$ is virtually G small.
    \item Every $M\in\catdbfl{R}$ is virtually G small.
\end{enumerate}
Note that 1 implies 2 by \cite[Theorem 1.2]{Hu/Yang/Zhu:2025}. To see that 3 implies 1, let $\underline{x}$ be a system of parameters for $\h{0}{R}$. By \cite[Lemma 7.12]{Bird/Shaul/Sridhar/Williamson:2025} there exists a dg-module $E$ such that $\rhom{R}{R/\!/\underline{x}}{E}$ has finite injective dimension and so that $\rhom{R}{R/\!/\underline{x}}{E}$ is in $\catdbfl{R}$. By assumption there is a dg-module $F_{M}$ in $\thick{R}{}{\rhom{R}{R/\!/\underline{x}}{E}}$ that has finite Gorenstein dimension. Since $F_{M}$ also has finite injective dimension, $R$ is Gorenstein by \cite[Proposition 3.2]{Soto/Sridhar:2025}.
\end{remark}

\begin{lemma} \label{L_RhomSRwithRhomSR} Let $\varphi\colon R\rightarrow S$ be a morphism of dg-algebras with $\gdim{R}{S}<\infty$. Then there is a isomorphism in $\catd{S}$
\[\rhom{S}{\rhom{R}{S}{R}}{\rhom{R}{S}{R}} \simeq S.\]
\end{lemma}

\begin{proof}
Just note that we have
\begin{align*}
\rhom{S}{\rhom{R}{S}{R}}{\rhom{R}{S}{R}} &\simeq \rhom{R}{\rhom{R}{S}{R}\ltime{}{S}{S}}{R} \\
&\simeq \rhom{S}{S}{\rhom{R}{\rhom{R}{S}{R}}{R}} \\
&\simeq \rhom{S}{S}{S} \\
&\simeq S,
\end{align*}
where the first and second isomorphisms are in $\catd{S}$ by \cite[Proposition 12.10.12]{Yekutieli:2020}, and the third isomorphism holds by the same reasoning in \cite[Proposition 12.3.19]{Christensen/Foxby/Holm:2024} since $S$ is derived reflexive over $R$.
\end{proof}

The following was proved in \cite[Theorem 6.2]{Dwyer/Greenlees/Iyengar:2006} when $R$ is a noetherian local ring.
\begin{proposition} \label{P_DwyerGreenleesIyengarTheorem} Let $R$ be a dg-algebra with $R_{0}$ a noetherian local ring and let $M\in\catdbf{R}$ be a dg-module. If $\rhom{R}{M}{M}$ has finite projective dimension and $M$ is virtually $G$ small, then $M$ has finite projective dimension. Moreover, if $\rhom{R}{M}{M}$ and $R$ are isomorphic, then $M$ is isomorphic to a shift of $R$.
\end{proposition}

\begin{proof}
This result was proved in \cite[Theorem 6.2]{Dwyer/Greenlees/Iyengar:2006} when $R$ is a noetherian local ring. The exact same proof works in the dg-setting. One only needs to note that it is by \cite[Proposition B.7 and Corollary B.8.1]{Avramov/Iyengar/Nasseh/SatherWagstaff:2019} that having the equality $\pseries{M}{R}=t^{a}$ for some $a$ implies that $M$ is isomorphic to a shift of $R$; this is needed in the last paragraph of the proof of \cite[Theorem 6.2]{Dwyer/Greenlees/Iyengar:2006}.
\end{proof}

The following is new even when both $R$ and $S$ are rings. 
\begin{theorem} \label{T_GvirtuallysmallnessascendsimpliesquasiGor}
Let $\varphi\colon R\rightarrow S$ be a morphism of dg-algebras with $S_{0}$ a noetherian local ring and $\gdim{R}{S}<\infty$. If every $M\in \catdbfl{S}$ that has finite Gorenstein dimension over $R$ is virtually $G$ small over $S$, then $\varphi$ is quasi-Gorenstein. 
\end{theorem}

\begin{remark} Letting $\varphi\colon R\rightarrow S$ be a module finite local homomorphism of noetherian local rings with $\gdim{R}{S}<\infty$, a consequence of \cite[Theorem 6.7]{Dwyer/Greenlees/Iyengar:2006} and the proof of \cite[Theorem 9.12]{Dwyer/Greenlees/Iyengar:2006} is that if every $M\in \catdbf{S}$ that is virtually $G$ small over $R$ is also virtually $G$ small over $S$, then $\varphi$ is quasi-Gorenstein.
\end{remark}

\begin{proof}[Proof of \Cref{T_GvirtuallysmallnessascendsimpliesquasiGor}]
Let $\underline{x}$ be a system of parameters for $\h{0}{S}$ and let $E=S/\!/\underline{x}$. Then $E$ is a dg-algebra as in \Cref{S_dgAlgebraSetup} and $\h{0}{E}$ is artinian, see \Cref{C_KoszulObjectStuff}. Consider the morphisms of dg-algebras $R\xrightarrow{\varphi} S\xrightarrow{f}E$. Since
\[E\in\thick{S}{}{S}\subseteq \thick{R}{}{S},\]
we have $\gdim{R}{E}<\infty$ by \cite[Proposition 3.1]{Hu/Yang/Zhu:2025}, and so $\gdim{R}{\rhom{R}{E}{R}}<\infty$ by \Cref{L_checkwhenMhasfiniteGdim}. Since
\[\rhom{R}{E}{R} \simeq \rhom{R}{\ltime{E}{S}{S}}{R} \simeq \rhom{S}{E}{\rhom{R}{S}{R}},\]
$\rhom{R}{E}{R}$ is in $\catdbfl{S}$ by \cite[2.1.2 and Proposition 2.4]{Minamoto:2021a}, and so $\gdim{R}{\rhom{R}{E}{R}}$ being finite implies $\rhom{R}{E}{R}$ is virtually G small over $S$ by assumption. Now the following isomorphisms hold in $\catd{S}$
\begin{align*}
\rhom{S}{\rhom{R}{E}{R}}{\rhom{R}{E}{R}} &\simeq \rhom{R}{\rhom{R}{E}{R}\ltime{}{S}{E}}{R} \\ 
&\simeq \rhom{S}{E}{\rhom{R}{\rhom{R}{E}{R}}{R}} \\
&\simeq \rhom{S}{E}{E} \in \thick{S}{}{S},
\end{align*}
where the first and second isomorphisms are by \cite[Proposition 12.10.12]{Yekutieli:2020}, and the third isomorphism holds by the same reasoning as in \cite[Proposition 12.3.19]{Christensen/Foxby/Holm:2024} since $E$ is derived reflexive over $R$. This implies that $\projdim{S}{\rhom{R}{E}{R}}<\infty$ by \Cref{P_DwyerGreenleesIyengarTheorem}, and so $\gdim{S}{\rhom{R}{E}{R}}<\infty$ by \cite[Proposition 4.1]{Hu/Yang/Zhu:2025}. Since $\rhom{S}{E}{S}$ is isomorphic to a shift of $E$ and since $\gdim{S}{E}<\infty$, $f$ is quasi-Gorenstein, and so $\gdim{E}{\rhom{R}{E}{R}}<\infty$ by \Cref{P_FiniteGdimAcrossMorphisms}. Now by \Cref{L_RhomSRwithRhomSR} we have an isomorphism
\[\rhom{E}{\rhom{R}{E}{R}}{\rhom{R}{E}{R}} \simeq E,\]
and so $\rhom{R}{E}{R}$ is isomorphic to a shift of $E$ by \Cref{P_DwyerGreenleesIyengarTheorem}. This implies $f\varphi$ is quasi-Gorenstein, giving the first isomorphism below for some $a$ 
\begin{align*}
\Sigma^{a}E \simeq \rhom{R}{E}{R} \simeq \rhom{R}{\ltime{E}{S}{S}}{R} &\simeq \rhom{S}{E}{\rhom{R}{S}{R}} \\
&\simeq \rhom{S}{E}{S}\otimes_{S}^{\mathsf{L}} \rhom{R}{S}{R} \\
&\simeq \Sigma^{b}E\otimes_{S}^{\mathsf{L}}\rhom{R}{S}{R}
\end{align*}
for some $b$, where the fourth isomorphism is by \cite[Lemma 2.7]{Bird/Shaul/Sridhar/Williamson:2025}. Since $\gdim{R}{S}<\infty$, $\rhom{R}{S}{R}$ is in $\catdbf{S}$, and so by \cite[Proposition B.7]{Avramov/Iyengar/Nasseh/SatherWagstaff:2019} we can take a minimal semifree resolution $F$ of $\rhom{R}{S}{R}$ over $S$. We now have $\Sigma^{a}E \simeq \Sigma^{b}E\otimes_{S}F$, and the proof of \cite[Lemma 1.5.3.a]{Avramov/Foxby:1997} gives $\pseries{\Sigma^{a}E}{S} = \pseries{\Sigma^{b}E}{S}\pseries{F}{S}$ which implies $\pseries{F}{S}=t^{n}$ for some $n$. By \cite[Corollary B.8.1]{Avramov/Iyengar/Nasseh/SatherWagstaff:2019} this forces $F$ to be isomorphic to a shift of $S$, and so $\rhom{R}{S}{R}$ is isomorphic to a shift of $S$, which means $\varphi$ is quasi-Gorenstein.
\end{proof}

\section{The canonical augmentation of the Koszul complex}
In this section, for a dg-algebra $R$, we study the quasi-Gorensteinness of the canonical augmentation $R\rightarrow\h{0}{R}$, focusing on when $R$ is a Koszul complex.

\begin{remark} \label{R_NegativeGorDim} Note that for a dg-algebra $R$, if the canonical augmentation $R \to \HR$ is quasi-Gorenstein, then $\amp\rhom{R}{\HR}{R} = 0$. Since $\sup\rhom{R}{\HR}{R} = \sup R$ by \cite[Proposition 3.3]{Shaul:2018}, we have that $\inf\rhom{R}{\HR}{R} = \sup R$. Thus, by the definition of Gorenstein dimension, we have that $\gdim{R}{\HR} = -\sup R$. This unintuitively implies that the Gorenstein dimension of $\HR$ with respect to $R$ is always nonpositive if it is finite.
\end{remark}

It is not always the case that $R$ and $\h{0}{R}$ are simultaneously Gorenstein, and so by \Cref{P_RGorAndSGorSimultaneously} it is not always the case that the canonical augmentation is quasi-Gorenstein. A consequence of \cite[Proposition 5]{Shaul:2022} is that it is always possible to write down a dg-algebra $R$ that is not Gorenstein such that $\h{0}{R}$ is Gorenstein. The next example shows that it is possible for a dg-algebra $R$ to be Gorenstein without $\h{0}{R}$ being Gorenstein.
\begin{example} Let $A=k[[x,y]]$ where $k$ is a field and let $R$ be the Koszul complex on $x^{2},xy$. Then $A$ is Gorenstein, and so $R$ is Gorenstein by \cite[Theorem 4.9]{Frankild/Jorgensen:2003}, but $\h{0}{R}=A/(x^{2},xy)$ is not Gorenstein. Now by \Cref{P_RGorAndSGorSimultaneously} $R\rightarrow \h{0}{R}$ is not quasi-Gorenstein.
\end{example}

The following gives an example where the canonical augmentation is quasi-Gorenstein.
\begin{example} Let $A=k[[x,y]]/(xy)$ where $k$ is a field. Let $R$ be the Koszul complex on $x$, and consider the natural morphisms $A\xrightarrow{\psi} R\xrightarrow{\varphi} \h{0}{R}$. Since both $A$ and $\h{0}{R}$ are Gorenstein, $\varphi\psi$ is quasi-Gorenstein by \cite[7.7.2]{Avramov/Foxby:1997}. Also, $\psi$ is quasi-Gorenstein by \cite[Lemma 4.8]{Frankild/Jorgensen:2003}, and so $\varphi$ is quasi-Gorenstein by \Cref{P_PhiPsiPsiGorImpliesPhiGor}.
\end{example}

\begin{lemma} \label{L_TopAndBottomHomology} Let $R$ be a dg-algebra. If $R\rightarrow \h{0}{R}$ is quasi-Gorenstein, then $\h{0}{R}$ and $\h{\sup{R}}{R}$ are isomorphic as $\h{0}{R}$-modules.
\end{lemma}

\begin{proof}
The morphism $R\rightarrow \h{0}{R}$ being quasi-Gorenstein gives an isomorphism between $\rhom{R}{\h{0}{R}}{R}$ and a shift of $\h{0}{R}$, and \cite[Proposition 3.3]{Shaul:2018} gives an isomorphism
\[\h{\sup{R}}{\rhom{R}{\h{0}{R}}{R}}\cong \h{\sup{R}}{R},\]
and so $\h{0}{R}$ and $\h{\sup{R}}{R}$ must be isomorphic.
\end{proof}

\begin{proposition} \label{P_QuasiGorAnnihilator} Let $A$ be a noetherian local ring, let $\underline{x}=x_{1},\dots,x_{n}$ be a sequence in $A$, let $R$ be the Koszul complex on $\underline{x}$, and assume $\ann{A}{(\underline{x})}\neq 0$. If $R\rightarrow \h{0}{R}$ is quasi-Gorenstein, then $\ann{A}{(\underline{x})} \cong A/(\underline{x})$. Moreover, we have $\ann{A}\left({\ann{A}{(\underline{x})}}\right) = (\underline{x})$, $\gdim{R}{\HR}=-n$, and $\gdim{A}{A/(\underline{x})}=0$.
\end{proposition}

\begin{proof}
Since $\h{n}{R} \cong \ann{A}{(\underline{x})} \neq 0$, we have the equality $\sup{R}=n$, and so we have 
\[\ann{A}{(\underline{x})} \cong \h{n}{R} \cong \h{\sup{R}}{R} \cong \h{0}{R} \cong A/(\underline{x}),\]
where the third isomorphism is by \Cref{L_TopAndBottomHomology}.

Moreover, the isomorphism $\ann{A}{(\underline{x})} \cong A/(\underline{x})$ gives an equality $\ann{A}\left({\ann{A}{(\underline{x})}}\right) = (\underline{x})$ and \Cref{R_NegativeGorDim} gives the first equality below
\[\gdim{R}{\HR} = -\sup{R} = -n.\]
Finally, \Cref{P_FiniteGdimAcrossMorphisms} gives the first equality below
\[\gdim{A}{\HR} = \gdim{A}{R} + \gdim{R}{\HR} = n - n =0,\]
finishing the proof.
\end{proof}

\begin{remark}
The converse of \Cref{P_QuasiGorAnnihilator} is false. Take $A = k[[x, y, z]]/(xy, xz)$ and consider the sequence $y, z$. We have that $\ann{A}{(y, z)} = (x)$ and that $A/(y, z) \cong (x)$ as $A$-modules, and so $\ann{A}{(y, z)} \cong A/(y, z)$. However, $A$ is not Gorenstein and $A/(y, z)$ is Gorenstein, and so the map $A \to A/(y, z)$ cannot be quasi-Gorenstein. Thus, the map from the Koszul complex $R$ on $(y, z)$ to $\h{0}{R}$ cannot be quasi-Gorenstein by \Cref{P_CompositionQuasiGor}. 
\end{remark}

\begin{remark} Let $A$ be an artinian local ring that is not Gorenstein. Then there exists an ideal $I$ such that $\ann{A}\left({\ann{A}{I}}\right)\neq I$, see \cite[Exercise 3.2.15]{Bruns/Herzog:1993}. Let $R$ be the Koszul complex on a generating set for $I$. Since $A$ is artinian, every nonzero ideal has a nonzero annihilator, and so $R\rightarrow \h{0}{R}$ is not quasi-Gorenstein by \Cref{P_QuasiGorAnnihilator}. This gives examples where the canonical augmentation of the Koszul complex is not quasi-Gorenstein.
\end{remark}

Below we recall the definition of an exact sequence, see \cite{Dibaei/Gheibi:ExactZeroDivisors, Henriques/Sega:2011, Kielpinski/Simson/Tyc:1978}.
\begin{chunk} \label{C_ExactZeroDivisor} Let $A$ be a noetherian local ring. An element $x\in A$ is an exact zero divisor if it satisfies
\[A \neq \ann{A}{(x)} \cong A/(x) \neq 0.\]
This is equivalent to saying that there is a nonzero element $y \in A$ such that $\ann{A}{x}=(y)$ and $\ann{A}{y}=(x)$. A sequence $x_{1},\dots,x_{n}$ is a sequence of exact zero divisors if $x_{i}$ is an exact zero divisor on $A/(x_{1},\dots,x_{i-1})$ for $1\leq i\leq n$. The element $x$ is called exact if $x$ is $A$-regular or is an exact zero divisor. A sequence $x_{1},\dots,x_{n}$ is called exact if $x_{i}$ is exact on $A/(x_{1},\dots,x_{i-1})$ for $1\leq i\leq n$.
\end{chunk}

\begin{remark} \label{R_NonzeroAnnihilatorRemark} The assumption on the annihilator in \Cref{P_QuasiGorAnnihilator} happens often. Let $\underline{x}=x_{1},\dots,x_{n}$ be a sequence in a noetherian local ring $A$.
\begin{enumerate}
\item If the length of the longest $A$-regular sequence in $(\underline{x})$ is zero, then $\sup{R}=n$ by the depth sensitivity of the Koszul complex \cite[Theorem 1.6.17]{Bruns/Herzog:1993}. Thus, we have that $\ann{A}{(\underline{x})} \cong \h{n}{R} \neq 0$.
\item By \cite[Theorem 3.7]{Avramov/Henriques/Sega:2013} if $\underline{x}$ is a sequence of exact zero divisors, then the length of the longest $A$-regular sequence in $(\underline{x})$ is zero.
\end{enumerate}
\end{remark}

\begin{theorem} \label{P_ExactSequences} Let $A$ be a noetherian local ring, let $x_{1},\dots,x_{n}$ be a sequence in $A$, and let $R_{i}$ be the Koszul complex on $x_{1},\dots,x_{i}$. Then the following are equivalent.
\begin{enumerate}
    \item $A/(x_{1},\dots,x_{i-1})\rightarrow A/(x_{1},\dots,x_{i})$ is quasi-Gorenstein for $1\leq i\leq n$.
    \item $R_{i}\rightarrow \h{0}{R_{i}}$ is quasi-Gorenstein for $1\leq i\leq n$.
    \item $R_{n}\rightarrow \h{0}{R_{n}}$ is quasi-Gorenstein and $\gdim{A/(x_{1},\dots,x_{i-1})}{A/(x_{1},\dots,x_{i})}$ is finite for $1\leq i\leq n$.
    \item $x_{1},\dots,x_{n}$ is an exact sequence.
\end{enumerate}
\end{theorem}

\begin{proof}
The equivalence of (1) and (2) follows from Propositions \ref{P_PhiPsiPsiGorImpliesPhiGor} and \ref{P_CompositionQuasiGor}, and (2) implies (3) by definition. To see that (3) implies (1), consider the composition $A \rightarrow R_n \rightarrow \h{0}{R_n}$. The first arrow is quasi-Gorenstein by \cite[Lemma 4.8]{Frankild/Jorgensen:2003} and the second arrow is quasi-Gorenstein by assumption, and so $A \to \h{0}{R_n}$ is quasi-Gorenstein by \Cref{P_CompositionQuasiGor}. Therefore $A/(x_{1},\dots,x_{i-1})\rightarrow A/(x_{1},\dots,x_{i})$ is quasi-Gorenstein for all $1 \leq i \leq n$ by repeated applications of \cite[Theorem 8.10.b]{Avramov/Foxby:1997}.

We now show the equivalence of (1) and (4). We have (4) implies (1) by \cite[3.1 and Theorem 3.7]{Avramov/Henriques/Sega:2013} and \cite[Remark 8]{Garcia/Soto:2004}. We now show (1) implies (4). If $\ann{A/(x_{1},\dots,x_{i-1})}{(x_{i})}=0$, then $x_{i}$ is a nonzerodivisor on $A/(x_{1},\dots,x_{i-1})$. If $\ann{A/(x_{1},\dots,x_{i-1})}{(x_{i})}\neq 0$, then consider the composition
\begin{equation*}
    A/(x_1, \ldots, x_{i-1}) \longrightarrow K_{A/(x_1, \ldots, x_{i-1})}(x_{i}) \longrightarrow A/(x_1, \ldots, x_{i})
\end{equation*}
where $K_{A/(x_1, \ldots, x_{i-1})}(x_{i})$ is the Koszul complex on $x_{i}$ over $A/(x_1, \ldots, x_{i-1})$. The first arrow is quasi-Gorenstein by \cite[Lemma 4.8]{Frankild/Jorgensen:2003}, and the composition is quasi-Gorenstein by assumption. Thus, the second arrow is quasi-Gorenstein by \Cref{P_PhiPsiPsiGorImpliesPhiGor}, and so $\ann{A/(x_{1},\dots,x_{i-1})}{(x_{i})}$ and $A/(x_{1},\dots,x_{i-1},x_{i})$ are isomorphic by \Cref{P_QuasiGorAnnihilator}. This tells us that $x_{i}$ is an exact zero divisor by definition.
\end{proof}

The following example illustrates that all the maps $A/(x_{1},\dots,x_{i-1})\rightarrow A/(x_{1},\dots,x_{i})$ in \Cref{P_ExactSequences} must be quasi-Gorenstein for the sequence to be exact, and it is taken from \cite[Example 4.8]{Dibaei/Gheibi:ExactZeroDivisors}.
\begin{example}
Let $k$ be a field and let $A = k[[x, y, z]]/(x^2, y^2 + xz, z^2)$. Then $A$ is artinian and Gorenstein. The sequence $x, y, z$ is exact, as $A/(x)$, $A/(x, y)$, and $A/(x, y, z)$ are all Gorenstein. However, the map $A \to A/(y)$ is not quasi-Gorenstein, as $A/(y)$ is not Gorenstein. As expected, the sequence $y, x, z$ is not exact by \cite[Proposition 3.3]{Dibaei/Gheibi:ExactZeroDivisors}.
\end{example}

Now using \Cref{P_ExactSequences} we can build new quasi-Gorenstein morphisms of dg-algebras from old ones.
\begin{corollary} \label{C_DGAexampleWithExactElement} Let $R$ be a dg-algebra and let $x$ be in $\HR$. If $R\rightarrow \HR$ is quasi-Gorenstein and $x$ is exact, then $R/\!/x\rightarrow \h{0}{R/\!/x}$ is quasi-Gorenstein.
\end{corollary}

\begin{proof}
Consider the following commutative diagram
% https://q.uiver.app/#q=WzAsNCxbMCwwLCJBIl0sWzEsMCwiQiJdLFswLDEsIkMiXSxbMSwxLCJEIl0sWzAsMSwiZiJdLFsxLDMsInAiXSxbMCwyLCJnIiwyXSxbMiwzLCJoIl1d
\[\begin{tikzcd}
	R & \h{0}{R} \\
	R/\!/x & \h{0}{R/\!/x}\cong\h{0}{R}/(x)
	\arrow["f", from=1-1, to=1-2]
	\arrow["g"', from=1-1, to=2-1]
	\arrow["p", from=1-2, to=2-2]
	\arrow["h", from=2-1, to=2-2]
\end{tikzcd}\]
Since $x$ is exact, $p$ is quasi-Gorenstein by \Cref{P_ExactSequences}, and so $pf$ is quasi-Gorenstein by \Cref{P_CompositionQuasiGor}. This implies that $hg$ is quasi-Gorenstein. Since $R/\!/x$ has finite projective dimension and is isomorphic to a shift of $\rhom{R}{R/\!/x}{R}$, we have that $g$ is quasi-Gorenstein, and so $h$ is quasi-Gorenstein by \Cref{P_PhiPsiPsiGorImpliesPhiGor}.
\end{proof}

\section{More (non)examples of quasi-Gorenstein morphisms} \label{S_MoreExamples}
We now provide more examples of dg-algebra morphisms, some of which are quasi-Gorenstein and some of which are not.

\subsection{dg-algebra resolutions}
\Cref{P_ExactSequences} gives examples where $R\rightarrow \HR$ is quasi-Gorenstein and $R$ is a Koszul complex. In \Cref{P_NontrivialQuasiGorDGMorphism} we give a way of producing dg-algebras $R$ that are not Koszul complexes such that $R\rightarrow \HR$ is quasi-Gorenstein. In the following proposition, when we say a homomorphism is Gorenstein, we mean it is quasi-Gorenstein and has finite projective dimension.
\begin{proposition} \label{P_NontrivialQuasiGorDGMorphism} Let $A$ be a noetherian local ring, let $I\subseteq J$ be ideals in $A$ such that $A\rightarrow A/I$ is Gorenstein and $A\rightarrow A/J$ is quasi-Gorenstein. If $R$ is a free resolution of $\ltime{A/I}{A}{A/J}$ that has a dg-algebra structure, then both $A\rightarrow R$ and $R\rightarrow \HR$ are quasi-Gorenstein and the amplitude of $R$ is positive. 
\end{proposition}

We prove the proposition after the following remark.
\begin{remark} \label{R_NontrivialQuasiGorDGMorphism} Here we give a way of finding explicit $A,I,J,R$ as in \Cref{P_NontrivialQuasiGorDGMorphism}.
\begin{enumerate}
    \item Under the setup of \Cref{P_NontrivialQuasiGorDGMorphism}, one would expect $\ltime{A/I}{A}{A/J}$ to have infinite projective dimension over $A$. This is because $\ltime{A/I}{A}{A/J}$ has finite projective dimension if and only if $A/J$ has finite projective dimension by \cite[Proposition 16.4.17]{Christensen/Foxby/Holm:2024}, and there exist quasi-Gorenstein ideals $J$ of infinite projective dimension.
    \item Let $(A,\mfm,k)$ be a Gorenstein local ring that is not regular and let $I$ be an ideal such that $A\rightarrow A/I$ is Gorenstein. Since $A$ is Gorenstein, $A\rightarrow k$ is quasi-Gorenstein, and so $I\subseteq \mfm$ are ideals that satisfy the assumptions in \Cref{P_NontrivialQuasiGorDGMorphism}. Since $A$ is not regular, $k$ has infinite projective dimension. Thus, $\ltime{A/I}{A}{k}$ has infinite projective dimension by the paragraph above, which implies $\ltime{A/I}{A}{k}$ cannot be isomorphic to a Koszul complex. See \cite[Example 3.2]{Jorgensen:2003} for an explicit example of a noetherian local ring $A$ and an ideal $I$ such that both $A$ and $A/I$ are Gorenstein but not regular, $A/I$ has finite projective dimension, and $I$ is not generated by a regular sequence. 
    
    Writing down a Gorenstein ideal that is not generated by a regular sequence is far from straightforward. We now explain a way to search for such an ideal. By the Buchsbaum-Eisenbud structure theorem \cite[Theorem 2.1]{Buchsbaum/Eisenbud:1977} an ideal $I$ of grade 3 in a noetherian local ring $A$ is Gorenstein if and only if $I$ is the ideal of $(n-1)$th order pfaffians of some $n\times n$ alternating matrix of rank $n-1$. Moreover, such an $I$ is minimally generated by $n$ elements. It is easy to see that a Gorenstein ideal of grade 1 is generated by a nonzerodivisor, and it is a consequence of the Hilbert-Burch theorem that a Gorenstein ideal of grade 2 is generated by a regular sequence, see \cite[Theorem 1.4.17]{Bruns/Herzog:1993}.

    Let $I$ be an ideal in a Cohen-Macaulay local ring $A$. Then the height of $I$ is at most the minimal number of generators of $I$, and the two are equal if and only if $I$ is generated by a regular sequence. This implies that a Gorenstein ideal of grade 3 coming from an $n\times n$ alternating matrix where $n>\dim{A}$ cannot be generated by a regular sequence. One can use Macaulay2 \cite{Grayson/Stillman:Macaulay2} to compute pfaffians of matrices and the grade of an ideal.
    \item Such an $R$ in \Cref{P_NontrivialQuasiGorDGMorphism} exists. To see this, let $I$ and $J$ be ideals in a ring $A$. Then by \cite[Proposition 2.1.10]{Avramov/1998} both $A/I$ and $A/J$ have a free resolution with a dg-algebra structure. Now take $R$ to be the tensor product of two such resolutions. 
    
    Note that it follows from \cite[Proposition 1.3]{Buchsbaum/Eisenbud:1977} that the minimal free resolution of a cyclic module of projective dimension at most three can be given a dg algebra structure, but by \cite[Theorem 2.3.1]{Avramov/1998} there exists a cyclic module whose minimal free resolution admits no dg-algebra structure. This tells us that if both $I$ and $J$ are Gorenstein ideals of grade at most 3, $R$ can be taken to be the tensor product of the minimal free resolutions of $A/I$ and $A/J$.
\end{enumerate}
\end{remark}

\begin{proof}[Proof of \Cref{P_NontrivialQuasiGorDGMorphism}]
We first show that $R\rightarrow \HR$ is quasi-Gorenstein. Consider
\[A\xrightarrow{f} R\xrightarrow{g} \HR\cong A/J.\]
Since $gf$ is quasi-Gorenstein, by \Cref{P_PhiPsiPsiGorImpliesPhiGor} it is enough to show $f$ is quasi-Gorenstein. Since $A/I$ has finite projective dimension and since $A/J$ has finite Gorenstein dimension, $\ltime{A/I}{A}{A/J}$ has finite Gorenstein dimension. Also, the following isomorphisms hold:
\begin{align*}
    \rhom{A}{R}{A} \simeq \rhom{A}{\ltime{A/I}{A}{A/J}}{A} &\simeq \rhom{A}{A/I}{\rhom{A}{A/J}{A}} \\
    &\simeq \ltime{\rhom{A}{A/I}{A}}{A}{\rhom{A}{A/J}{A}} \\
    &\sim \ltime{A/I}{A}{A/J} \\
    &\simeq R.
\end{align*}
The third isomorphism is by \cite[Corollary 12.3.23]{Christensen/Foxby/Holm:2024} since $A/I$ has finite projective dimension. This shows that $f$ is quasi-Gorenstein. To see that the amplitude of $R$ is positive, just note that $\tor{A}{1}{A/I}{A/J}\neq0$ since $IJ\neq I\cap J$ by Nakayama's lemma.
\end{proof}

\subsection{Trivial extensions} \label{S_TrivialExtension} In this subsection we give a way of producing morphisms of dg-algebras of positive amplitude, some of which are quasi-Gorenstein and some of which are not, by using trivial extensions. For a dg-algebra $R$ and a dg-module $M$, the trivial extension of $R$ by $M$, denoted by $R\ltimes M$, is the dg-module $R\oplus M$ with multiplication defined by
\[(r_{1},m_{1})\cdot (r_{2},m_{2}) := (r_{1}r_{2},r_{1}m_{2}+m_{1}r_{2}).\]
By \cite[Proposition 7.1]{Shaul:2020} if $M$ is in $\catdbf{R}$ and $\inf{M}>0$, then $\h{0}{R\ltimes M} \cong \HR$ and $R\ltimes M$ is a dg-algebra as in \Cref{S_dgAlgebraSetup}. For more on trivial extensions in the dg-algebra setting, see \cite{JorgensenPeter:2003} and \cite[Section 7]{Shaul:2020}.

It is a well-known fact in commutative algebra that a noetherian local ring has a canonical module if and only if R is the homomorphic image of a Gorenstein local ring. This fact has been extended by J\o rgensen to show that a noetherian local ring has a dualizing complex if and only if it is the homomorphic image of a Gorenstein dg-algebra \cite[Theorem 2.2]{JorgensenPeter:2003}. We now adapt J\o rgensen's proof to show that a dg-algebra has a dualizing dg-module if and only if it is the homomorphic image of a Gorenstein dg-algebra.
\begin{proposition} \label{P_ZachJorgensenResult}
Let $R$ be a dg-algebra and let $D$ be a dg $R$-module with $\inf{D} > 0$. Then $D$ is a dualizing dg $R$-module if and only if $R \ltimes D$ is a Gorenstein dg-algebra.
\end{proposition}

\begin{proof}
First assume $D$ is a dualizing dg $R$-module. In the canonical morphism $R \to R \ltimes D$, the induced homomorphism on the zeroth homologies is an isomorphism. Thus, by \cite[Proposition 7.5]{Yekutieli:2016}, we have that $\rhom{R}{R \ltimes D}{D}$ is a dualizing dg ($R \ltimes D$)-module. Now a similar argument as in \cite[Theorem 2.2]{JorgensenPeter:2003} shows that $R \ltimes D$ is a Gorenstein dg-algebra.

Now suppose that $R \ltimes D$ is a Gorenstein dg-algebra. Then $R \ltimes D$ is a dualizing dg ($R \ltimes D$)-module. In the canonical morphism $R \ltimes D \to R$, the induced homomorphism on the zeroth homologies is an isomorphism. Thus, by \cite[Proposition 7.5]{Yekutieli:2016}, we have that $\rhom{R \ltimes D}{R}{R \ltimes D}$ is a dualizing dg $R$-module.
\end{proof}

For the next result, recall that if $R$ is homologically $\mathfrak{m}$-adically complete, then $R$ has a dualizing dg-module by \cite[Proposition 7.21]{Shaul:2018}.
\begin{corollary} \label{C_UsingTrivialExtensions} Let $R$ be a dg-algebra. Then the following hold.
\begin{enumerate}
    \item Assume $D$ is a dualizing dg-module with $\inf{D}>0$. Then $R\rightarrow R\ltimes D$ is quasi-Gorenstein if and only if $R$ is Gorenstein.
    \item The canonical morphism $R\rightarrow R\ltimes \Sigma^n R$ is quasi-Gorenstein.
\end{enumerate}
\end{corollary}

\begin{proof}
We first prove (1). By \Cref{P_ZachJorgensenResult} $R\ltimes D$ is Gorenstein, and so $R\rightarrow R\ltimes D$ being quasi-Gorenstein implies $R$ is Gorenstein by \Cref{P_RGorAndSGorSimultaneously}. Now if $R$ is Gorenstein, then $R$ is a dualizing dg $R$-module, and so $\rhom{R}{R\ltimes D}{R}$ is a dualizing dg $(R\ltimes D)$-module by \cite[Proposition 7.5]{Yekutieli:2016}. Since dualizing dg modules are unique in the derived category up to shift by \cite[Corollary 7.16]{Yekutieli:2016} and since $R\ltimes D$ is Gorenstein by \Cref{P_ZachJorgensenResult}, $R\rightarrow R\ltimes D$ is quasi-Gorenstein.

We now prove (2). The following isomorphisms hold in $\catd{R}$
\begin{align*}
    \rhom{R}{R\ltimes \Sigma^n R}{R} \simeq \rhom{R}{R\oplus \Sigma^n R}{R} \simeq R\oplus \Sigma^{-n}R &\simeq \Sigma^{-n}(R\oplus \Sigma^n R) \\
    &\simeq \Sigma^{-n}(R\ltimes \Sigma^n R).
\end{align*}
This shows that $\rhom{R}{R\ltimes \Sigma^n R}{R} \sim (R \ltimes \Sigma^n R)$ in $\cat{R}{D}$, and a similar argument as \cite[Theorem 2.2]{JorgensenPeter:2003} gives that this isomorphism arises from a morphism in $\cat{R \ltimes D}{D}$. Thus, the canonical morphism is quasi-Gorenstein.
\end{proof}

\section{Gorenstein projective modules over Koszul complexes} \label{S_GorProjModsKoszul}
In this section we give ways of producing nonfree Gorenstein projective dg modules over Koszul complexes. 
\begin{lemma} \label{L_QuotientByExactZeroDivisorGorProj} Let $A$ be a noetherian local ring, let $\underline{x}=x_{1},\dots,x_{n}$ be a sequence of exact zero divisors, and let $R$ be the Koszul complex on $\underline{x}$. Then we have the equalities $\amp{R}=n$, $\gdim{A}{A/(\underline{x})}=0$, and $\gdim{R}{A/(\underline{x})}=-n$.
\end{lemma}

\begin{proof}
Consider the composition $A\xrightarrow{f} R\xrightarrow{g} A/\underline{x}$. By \Cref{P_ExactSequences} $g$ is quasi-Gorenstein, and by \Cref{R_NonzeroAnnihilatorRemark}  $\h{n}{R}=\ann{A}{(\underline{x})}\neq 0$. This implies $\amp{R}=n$. Also, by \Cref{P_QuasiGorAnnihilator} we have $\gdim{A}{A/\underline{x}}=0$ and $\gdim{R}{A/(\underline{x})}=-n$. 
\end{proof}

\begin{proposition} \label{P_NontrivialGorProjExample} Let $A$ be a noetherian local ring, let $\underline{x}=x_{1},\dots,x_{n}$ be a sequence of exact zero divisors, and let $R$ be the Koszul complex on $\underline{x}$. If $M$ is a finitely generated $A/\underline{x}$-module with finite Gorenstein dimension, then the following equalities hold.
\begin{enumerate}
    \item $\amp{\Sigma^{n-\gdim{A}{M}}M\oplus \Sigma^{-\gdim{A}{M}}M}=\amp{R}$.
    \item $\gdim{R}{\Sigma^{n-\gdim{A}{M}}M\oplus \Sigma^{-\gdim{A}{M}}M}=0$.
\end{enumerate}
\end{proposition}

\begin{proof}
Consider the composition $A\xrightarrow{f} R\xrightarrow{g} A/\underline{x}$. For (1), just note that \Cref{L_QuotientByExactZeroDivisorGorProj} gives the second equality below
\[\amp{\Sigma^{n-\gdim{A}{M}}M\oplus \Sigma^{-\gdim{A}{M}}M}=n=\amp{R}.\]
For (2), it is by \Cref{L_QuotientByExactZeroDivisorGorProj} that we have $\gdim{A}{A/(\underline{x})}=0$. Also, $f$ is quasi-Gorenstein by \cite[Lemma 4.8]{Frankild/Jorgensen:2003}, which implies $gf$ is quasi-Gorenstein by \Cref{P_CompositionQuasiGor}. Now \cite[Theorem 7.11]{Avramov/Foxby:1997} gives the first equality below
\[\gdim{A}{M} = \gdim{A}{A/\underline{x}} + \gdim{A/\underline{x}}{M} = \gdim{A/\underline{x}}{M},\]
and so $\gdim{A}{\Sigma^{n-\gdim{A}{M}}M\oplus \Sigma^{-\gdim{A}{M}}M}=n$. Therefore, $f$ being quasi-Gorenstein gives the first equality below by \Cref{P_FiniteGdimAcrossMorphisms}
\begin{align*}
    \gdim{R}{\Sigma^{n-\gdim{A}{M}}M\oplus \Sigma^{-\gdim{A}{M}}M} &= \gdim{A}{\Sigma^{n-\gdim{A}{M}}M\oplus \Sigma^{-\gdim{A}{M}}M} \\
    & \hspace{1.5em}- \gdim{A}{R} \\
    &= 0.\qedhere
\end{align*}
\end{proof}

\begin{lemma} \label{L_ComputeRhomOverRing} Let $A$ be a noetherian local ring, let $R$ be a Koszul complex on $\underline{x}=x_1, \dotsc, x_n$, and let $M$ be a dg $R$-module. Then
\[\rhom{R}{M}{R} \simeq \Sigma^{n}\rhom{A}{M}{A}.\]
\end{lemma}

\begin{proof}
Just note that we have
\begin{align*}
    \rhom{R}{M}{R} \simeq \rhom{R}{M}{\Sigma^{n}\rhom{A}{R}{A}} &\simeq \Sigma^{n}\rhom{A}{\ltime{M}{R}{R}}{A} \\
    &\simeq \Sigma^{n} \rhom{A}{M}{A}.\qedhere
\end{align*}
\end{proof}

Recall that by \Cref{R_NonzeroAnnihilatorRemark}, if $\underline{x}$ is a sequence of exact zero divisors, then $\ann{A}{(\underline{x})}\neq 0$.
\begin{proposition} \label{P_NonzeroAnnihilatorGorensteinProjective} Let $A$ be a Gorenstein local ring and let $R$ be a Koszul complex on a sequence $\underline{x}=x_{1},\dots,x_{n}$. If $M$ is a finitely generated Gorenstein projective $A$-module, then $\gdim{R}{\ltime{M}{A}{R}}=0$. Moreover, if $\ann{A}{(\underline{x})}\neq 0$ and $(0:_{M}(\underline{x}))\neq 0$, then the equality $\amp{\ltime{M}{A}{R}} =\amp{R}$ holds.
\end{proposition}

\begin{proof}
Since $A$ is Gorenstein, $R$ is Gorenstein by \cite[Theorem 4.8]{Frankild/Jorgensen:2003}, and so $\ltime{M}{A}{R}$ has finite Gorenstein dimension over $R$. We now compute the Gorenstein dimension of $\ltime{M}{A}{R}$. We have the following isomorphisms
\begin{align*}
    \rhom{R}{\ltime{M}{A}{R}}{R} \simeq \Sigma^{n}\rhom{A}{\ltime{M}{A}{R}}{A} &\simeq \Sigma^{n}\rhom{A}{M}{\rhom{A}{R}{A}} \\
    &\simeq \Sigma^{n}\rhom{A}{M}{\Sigma^{-n}R} \\
    &\simeq \Sigma^{n-n}\rhom{A}{M}{R} \\
    &\simeq \Hom{A}{M}{R},
\end{align*}
where the first isomorphism is by \Cref{L_ComputeRhomOverRing} and the fifth isomorphism is by \cite[Theorem 2.8]{Christensen/Frankild/Holm:2006} since $M$ is Gorenstein projective and $R$ is a complex of free $A$ modules. Note that 
\[\Hom{A}{M}{R} = \cdots\rightarrow \Hom{A}{M}{A}\rightarrow 0\]
where $\Hom{A}{M}{A}$ is in degree zero and we have an isomorphism
\[\h{0}{\Hom{A}{M}{R}} \cong \Hom{A}{M}{A}/(\underline{x})\Hom{A}{M}{A}.\]
Since $M$ is Gorenstein projective over $A$, $M$ is reflexive, and so $\Hom{A}{M}{A}$ is nonzero, which implies $\Hom{A}{M}{A}/(\underline{x})\Hom{A}{M}{A}$ is nonzero by Nakayama's Lemma. Therefore
\[\gdim{R}{\ltime{M}{A}{R}} = -\inf{\rhom{R}{\ltime{M}{A}{R}}{R}} = -\inf{\Hom{A}{M}{R}} = 0.\]
Moreover, we have
\[\h{n}{R} \cong \ann{A}{(\underline{x})}\neq 0 \quad\quad\text{and}\quad\quad \h{n}{\ltime{M}{A}{R}} \cong (0:_{M}(\underline{x}))\neq 0,\]
which implies $\amp{\ltime{M}{A}{R}} = n = \amp{R}$, finishing the proof.
\end{proof}

The following example shows that rings and modules as in \Cref{P_NonzeroAnnihilatorGorensteinProjective} exist.
\begin{example}
Let $k$ be a field, let $A=k[[x,y]]/(xy)$, and let $M=A/(x)$. Then both $\ann{A}{(x)}$ and $(0:_{M}(x))$ are nonzero. Since the depth of $M$ is one and since $A$ is a one dimensional Gorenstein local ring, $M$ is Gorenstein projective. 
\end{example}

\begin{remark} \label{R_ExactZeroDivisorMCMdgModules} Let $A$ be a noetherian local ring, let $R$ be the Koszul complex on a sequence $\underline{x}$, and assume $R$ is local-Cohen-Macaulay; see \cite[Definition 4.2]{Shaul:2020} for the definition of local-Cohen-Macaulay. If $A$ is Cohen Macaulay, then $R$ is local-Cohen-Macaulay by \cite[Theorem 4.2]{Shaul:2021}, but the converse is not true, see \cite[Remark 4.13]{Shaul:2021}. By \cite[Corollary 1.5]{Hu/Yang/Zhu:2025} if $M$ is a Gorenstein projective dg-module over $R$ such that $\amp{M}=\amp{R}$, then $M$ is a maximal local-Cohen-Macaulay dg-module; see \cite[Definition 6.1 and Definition 6.4]{Shaul:2021} for the definition of maximal local-Cohen-Macaulay dg-module. This tells us that if $A$ is Cohen-Macaulay and if $\underline{x}$ is a sequence of exact zero divisors, then \Cref{P_NontrivialGorProjExample} and \Cref{P_NonzeroAnnihilatorGorensteinProjective} give examples of nonfree maximal local-Cohen-Macaulay dg-modules over local-Cohen-Macaulay dg-algebras.
\end{remark}

\bibliographystyle{amsplain}
\bibliography{references}

\end{document}